\documentclass[12pt,oneside,english]{amsart}
\usepackage[T1]{fontenc}
\usepackage[utf8]{inputenc}
\usepackage{geometry}
\usepackage{amsthm}
\usepackage{amssymb}
\usepackage{mathdots}
\usepackage{esint}
\usepackage{float}
\usepackage{xcolor}

\newcommand{\chg}[1]{{\color{magenta}#1}}

\makeatletter
\numberwithin{equation}{section}
\numberwithin{figure}{section}
\theoremstyle{plain}
\newtheorem{thm}{\protect\theoremname}
\newtheorem{rmk}{\protect\remarkname}
\newtheorem{lem}{Lemma}[section]
\DeclareMathOperator{\Ai}{Ai}

\usepackage[english]{babel}
\usepackage{fancyhdr}
\usepackage{tikz}
\usetikzlibrary{arrows,shapes}
\usetikzlibrary{calc,decorations.markings}
\tikzset{
  reversed with radius/.style={
    x radius=#1,
    y radius=-#1,
 }
}

\tikzset{
  with arrows/.style={
    decoration={ markings,
      mark=between positions #1 and .999 step #1 with {\arrow{stealth}}
    }, postaction={decorate}
  }, with arrows/.default=25mm,
} 

\usepackage{scalerel}
\usepackage{stackengine}\stackMath
\newcommand{\reallywidehat}[1]{%
\savestack{\tmpbox}{\stretchto{%
  \scaleto{%
    \scalerel*[\widthof{\ensuremath{#1}}]{\kern-.6pt\bigwedge\kern-.6pt}%
    {\rule[-\textheight/2]{1ex}{\textheight}}
  }{\textheight}%
}{0.5ex}}%
\stackon[1pt]{#1}{\tmpbox}%
}

\providecommand{\theoremname}{Theorem}

\newcommand{\abs}[1]{\ensuremath{|#1|}}

\newcommand{\Abs}[1]{\ensuremath{\left|#1\right|}}

\renewcommand{\d}[1]{\ensuremath{\textnormal{d}#1}}

\newcommand{\cC}{\mathcal{C}}

\newcommand{\cO}{\mathcal{O}}

\usepackage{babel}

  \providecommand{\remarkname}{Remark}
\providecommand{\theoremname}{Theorem}

\usepackage[section]{placeins}

\usepackage{babel}

\providecommand{\theoremname}{Theorem}

\usepackage{babel}

\providecommand{\remarkname}{Remark}
\providecommand{\theoremname}{Theorem}

\usepackage{babel}

\providecommand{\remarkname}{Remark}
\providecommand{\theoremname}{Theorem}

\usepackage{babel}

\providecommand{\remarkname}{Remark}
\providecommand{\theoremname}{Theorem}

\usepackage{babel}

\providecommand{\remarkname}{Remark}
\providecommand{\theoremname}{Theorem}

\usepackage{babel}

\providecommand{\remarkname}{Remark}
\providecommand{\theoremname}{Theorem}

\makeatother

\usepackage{babel}
\providecommand{\theoremname}{Theorem}

\begin{document}

\title[Asymptotic sharpness of a Nikolskii type inequality]{Asymptotic sharpness of a Nikolskii type inequality for rational
functions in the Wiener algebra}

\author{Benjamin Auxemery}

\address{Aix-Marseille University, CNRS, I2M, Marseille,
France.}

\email{benjamin.auxemery@gmail.com}

\author{Alexander Borichev}

\address{Aix-Marseille University, CNRS, I2M, Marseille,
France.}

\email{alexander.borichev@math.cnrs.fr}

\author{Rachid Zarouf}

\address{Aix-Marseille University, University of Toulon, CNRS, CPT, Marseille, France}

\address{Aix-Marseille University, Laboratory ADEF, Marseille, France}

\email{rachid.zarouf@univ-amu.fr, rzarouf@gmail.com}

\keywords{Wiener algebra, Hardy space, rational functions, Fourier coefficients,
powers of a Blaschke factor, method of the stationary phase.}

\subjclass[2020]{30A10, 42A16, 30H10, 30H50, 30J10}

\begin{abstract}
We establish the asymptotic sharpness of a Nikolskii type inequality
proved by A. Baranov and R. Zarouf for rational functions $f$ in the
Wiener algebra of absolutely convergent Fourier series, with at most
$n$ poles, all lying outside the dilated disc $\frac{1}{\lambda}\mathbb{D}$,
where $\mathbb{D}$ denotes the open unit disc and $\lambda\in[0,1)$
is fixed. More precisely, this inequality tells that the Wiener norm of such
functions is bounded by their $H^{2}$-norm -- i.e., their norm in
the Hardy space of the disc -- times a factor of order $\sqrt{\frac{n}{1-\lambda}}$.
In this paper, we construct explicit test functions showing that this
bound cannot be improved in general: the inequality is asymptotically
sharp as $n\to\infty$, up to a universal constant, for every fixed
$\lambda\in[0,1)$.
\end{abstract}

\thanks{The second author was supported by ANR-24-CE40-5470}

\maketitle

\section{Introduction}

Let $\mathcal{P}_{n}$ be the complex space of polynomials of degree
less than or equal to $n\ge1$, and let $\mathbb{D}=\left\{ z\in\mathbb{C}\,:\left|z\right|<1\right\} $
be the unit disc in the complex plane. Given $\lambda\in[0,\,1),$
we define 
\[
\mathcal{R}_{n,\,\lambda}=\Bigl\{ \frac{p}{q}\,:\;p,\,q\in\mathcal{P}_{n},\, q \text{\ zero free in\ } \frac{1}{\lambda} \mathbb{D}  \Bigr\} ,
\]
the set of all rational functions 
of bidegree at most $(n,n)$, with all poles
outside of $\frac{1}{\lambda}\mathbb{D}.$
In particular, for $\lambda=0$, we get $\mathcal{R}_{n,\,0}=\mathcal{P}_n$.

A.~Baranov and R.~Zarouf proved in \cite{BZ} an S.~M.~Nikolskii type inequality
in the classical Beurling--Sobolev spaces 
\[
\ell_{A}^{p}=\Bigl\{ f\in\mathrm{Hol}(\mathbb{D}): f(z)=\sum_{k\ge 0}\hat{f}(k)z^{k},\,\|f\|_{\ell_A^p}^p=\sum_{k\ge 0}\abs{\hat{f}(k)}^{p}<\infty\Bigr\} ,
\]
for functions in $\mathcal{R}_{n,\,\lambda}$, 
in the case
where $p=1$ on the left-hand side and $p=2$ on the right-hand side
of the inequality. More precisely, they proved
that the embedding of the Wiener algebra $\ell_{A}^{1}=W$ into the Hardy
space $\ell_{A}^{2}=H^{2}$ is invertible on $\mathcal{R}_{n,\,\lambda}$
and that the corresponding embedding constant is $\lesssim\sqrt{\frac{n}{1-\lambda}}$.

\begin{thm}
\label{Nik_type_ineq} \textup{\cite[Theorem 2]{BZ} }Let $n\ge 1$,
$\lambda\in[0,1)$, and let $f\in\mathcal{R}_{n,\,\lambda}$. For some absolute constant $K$ we have 
\[
\|f\|_{\ell_{A}^{1}}\le K\sqrt{\frac{n}{1-\lambda}}\|f\|_{\ell_{A}^{2}}.
\]
\end{thm}

This inequality provides a quantitative control of the $\ell_{A}^{1}$-norm
by the $H^{2}$-norm on the class $\mathcal{R}_{n,\,\lambda}$ of
rational functions of bidegree at most $(n,n)$ with poles outside the dilated
disc $\frac{1}{\lambda}\mathbb{D}$. While the sharpness of the constant
is not yet established (this is the purpose of the present note),
it already plays a central role in \cite{BZ}, where it is used to
derive sharp estimates for certain Nevanlinna--Pick type interpolation
quantities constrained by Beurling--Sobolev norms, see \cite[Theorem 1]{BZ}.
Specifically, the authors exploit this embedding to control the $H^{\infty}$-norm
of a rational interpolant in terms of the input/output data associated
with the corresponding Nevanlinna--Pick type problem. 

The upper estimate in Theorem \ref{Nik_type_ineq} reflects the delicacy
of controlling the $\ell_A^1$-norm as $n$ grows large and as poles
approach the unit circle, and its utility goes beyond interpolation,
potentially impacting rational approximation problems and norm estimates
in related analytic function spaces.

Below we prove the sharpness of the embedding constant $\sqrt{\frac{n}{1-\lambda}}$
in the following sense. Let $b_\lambda(z)=(z-\lambda)/(1-\overline{\lambda}z)$.
Given $n\ge 1$ and $\lambda\in[1/2,1)$, we consider the test function $f_{n,\lambda}\in\mathcal{R}_{n,\,\lambda}$ defined by 
\[
f_{n,\lambda}(z)=\frac{1}{1-\lambda z}b_{\lambda}^{n-1}(z).
\]
Given $n\ge 1$ and $\lambda\in[0,1/2]$, we set 
\[
D_{n}(z)=\sum_{k=0}^{n-1}z^{k}\in\mathcal{R}_{n,\,0}\subset\mathcal{R}_{n,\,\lambda}
\]
be the (somewhat modified) standard Dirichlet kernel of order $n$.

\begin{thm}
\label{Nik_type_sharpness} There exists an absolute positive constant $K$ such that if $\lambda\in[1/2,1)$, 
then 
\[
\|f_{n,\lambda}\|_{\ell_{A}^{1}}\ge K\sqrt{\frac{n}{1-\lambda}}\|f_{n,\lambda}\|_{\ell_{A}^{2}},\qquad n\ge n(\lambda).
\]
If $n\geq1$ and $\lambda\in[0,1/2]$, then
\[
\|D_{n}\|_{\ell_{A}^{1}}\ge \frac1{\sqrt{2}}\sqrt{\frac{n}{1-\lambda}}\|D_{n}\|_{\ell_{A}^{2}}.
\]
\end{thm}

$\,$

We proceed with two remarks.

\begin{rmk}
The first part of Theorem \ref{Nik_type_sharpness}, whose proof is
delicate, shows the asymptotic sharpness (up to a numerical prefactor)
of the embedding coefficient $\sqrt{\frac{n}{1-\lambda}}$ in Theorem
\ref{Nik_type_ineq} as $n\to\infty$ for any fixed $\lambda\in[1/2,1)$.
The second part of Theorem \ref{Nik_type_sharpness} (the case $\lambda\in[0,1/2]$)
is trivial since $\|D_n\|_{\ell_A^1}=\sqrt{n}\|D_n\|_{\ell_A^2}$
and shows the sharpness of the inequality from Theorem \ref{Nik_type_ineq}
also in this case. 
\end{rmk}

\begin{rmk}
This asymptotic lower bound is of the same nature as a result established
in \cite{ZR}, where the sharp growth rate for a Bernstein-type inequality
in the Hardy space $H^{2}$ was proved for rational functions with
prescribed pole locations. In both settings, the goal is to determine
the asymptotic behavior, as $n\to\infty$, along with the correct
dependence on the pole constraint parameter --- denoted by $\lambda$. 
In \cite{ZR}, the sharp constant behaves like $1/(1-\lambda)$,
while in the present setting it is $1/\sqrt{1-\lambda}$, reflecting
the difference between the norms involved: the norm of the differentiation
operator in the Hilbert space $H^{2}$, versus the norm of the embedding
$H^{2}\hookrightarrow W$ into the Wiener algebra. The present case
is notably more challenging, since the Wiener norm $\|f\|_{\ell_{A}^{1}}$
is defined as the sum over the moduli of the Taylor coefficients, and lacks
the Hilbertian structure of $H^{2}$. In particular, the orthogonality
of the Malmquist--Walsh basis in the model space $K_{B}=H^{2}\ominus BH^{2}$
\cite[p. 157]{NN}, where $B$ is a finite Blaschke product encoding
the location and multiplicity of the poles, played a central role
in \cite{ZR}, and is unavailable here. Still, it is worth noting that
the test function $f_{n,\lambda}=\frac{1}{1-\lambda z}b_{\lambda}^{n-1}$
used in our paper coincides with the $n$-th vector of the Malmquist--Walsh
basis associated with the Blaschke product $B(z)=b_{\lambda}^{n}(z)$,
thereby highlighting a structural similarity between the two problems.
\end{rmk}
Fix $\lambda\in[1/2,1)$ and put
\[
\alpha=\frac{1-\lambda}{1+\lambda}\in(0,1/3].
\]
Given $k$ and $n$ we set $t=k/n$. Next, we set 
\begin{gather*}
f_n(z)=\frac{1}{1-\lambda z}b_{\lambda}^n(z),\\
\Phi(z)=
\log\left(z^{-t}b_{\lambda}(z)\right),\label{Phi}
\end{gather*}
where $\log$ denotes 
the principal branch of the complex logarithm. 
To state the next
theorem, which is obtained as an application of the method of stationary
phase \cite{Erd}, as well as a uniform version of this method \cite{BlHa,BV,CFU},
we need to introduce an auxiliary function 
$$
h(s)=h_t(s)=
-i\Phi(e^{is})=\arg\left(\left(z^tb_{\lambda}(z)\right)_{\vert z=e^{is}}\right)\qquad s\in[0,\pi].
$$
Then
\begin{equation}
h''(s)=-\frac{2\lambda(1-\lambda^2)\sin s}{(1+\lambda^2-2\lambda\cos s)^2}.
\label{four4}
\end{equation}

Since $|z^tb_{\lambda}(z)|=1$ for $z\in\partial\mathbb{D}$,
we have 
$$
\widehat{f_n}(k)=\frac{1}{\pi}\Re \Bigl\{ \int_{0}^{\pi}\frac{1}{1-\lambda e^{is}}e^{inh
(s)}\,{\rm d}s\Bigr\}.
$$

Given $t\in(\alpha,\alpha^{-1})$, we 
set
\begin{equation}
q(t)=\frac{1+\lambda^2}{2\lambda}-\frac{1-\lambda^2}{2\lambda t},
\label{h1xx}
\end{equation}

so that $|q(t)|<1$, 
and define $\varphi(t)$, $0<\varphi(t)<\pi$, 
by the relation
\begin{equation}
\cos(\varphi(t))=q(t). \label{h1}
\end{equation}
Next, we define $\psi(t)$, $-\pi/2<\psi(t)<0$, by the relation  
\begin{equation}
1-\lambda e^{i\varphi(t)}
=|1-\lambda e^{i\varphi(t)}|e^{i\psi(t)}.
\label{h2}
\end{equation}
Finally, we consider the function 
\begin{equation}
F(t)=h_t(\varphi(t))=-i\log(e^{-it\varphi(t)}e^{i\varphi(t)-2i\psi(t)})=\varphi(t)-t\varphi(t)-2\psi(t).
\label{h3}
\end{equation}
 
We need the following result whose proof is rather technical and which is a natural modification of the proof of \cite[Theorem 2~(2)]{BFZ} whose purpose
is to compute the asymptotic expansion of $\widehat{b_{\lambda}^{n}}(k)$ when $t=k/n$ is fixed and $n$ grows large.

\begin{thm}
\label{thm:3} Let $\alpha<\beta<1$. If $k\in[\beta n,\beta^{-1}n]$, 
then 
\[
\widehat{f_n}(k)=\sqrt{\frac{2}{n(1-\lambda^{2})\pi}}\frac{\cos\left(nF(t)-\psi(t)-\pi/4\right)+o(1)}{[(\alpha^{-1}-t)(t-\alpha)]^{1/4}},\qquad n\to\infty, 
\]
where $t=k/n$, $F$ is defined by \eqref{h3}, and $\psi$ is defined by \eqref{h2}. The remainder term $o(1)$ is uniform in $t\in[\beta,\beta^{-1}]$.
\end{thm}


\subsection*{Outline of the paper}

The paper is organized as follows.  
In Section~2 we prove Theorem~2, which establishes the asymptotic sharpness of the Nikolskii-type inequality stated in Theorem~1.  
The proof is based on an explicit analysis of the $\ell_A^1$-norm of the test function $f_n$, using the asymptotic expansion of its Fourier coefficients.  
After inserting the asymptotic formula of Theorem~3, we sum the dominant terms over a carefully chosen interval of frequencies $k$ corresponding to a subinterval of $t=k/n \in [\alpha,\,\alpha^{-1}]$.  
The resulting series is a highly oscillatory "pseudo-Riemann sum", in which the cosine terms encode the oscillations of the phase.  
To control these oscillations, we rely on Weyl's equidistribution criterion, which guarantees that the phase behaves like an equidistributed sequence modulo~$2\pi$ and allows us to estimate the averaged contribution of the terms.  This leads to the desired lower bound, showing that the Nikolskii-type inequality is asymptotically sharp.

Section~3 is devoted to the proof of Theorem~3, which provides a precise asymptotic expansion of the Fourier coefficients $\widehat{f_n}(k)$ as $n\to\infty$.  
The proof relies on the method of stationary phase applied to the contour integral representation of $\widehat{f_n}(k)$.  
When $t=k/n$ remains in a compact subset of $(\alpha,\alpha^{-1})$, we apply the classical stationary phase theorem of Fedoryuk~\cite{F1}.  

\section{Proof of Theorem ~\ref{Nik_type_sharpness}}
\begin{proof}[Proof of Theorem~\ref{Nik_type_sharpness}]
To simplify the notation, we replace $n-1$ by $n$ and consider the test function 
\[
f_n(z)=\frac{1}{1-\lambda z}b_{\lambda}^{n}(z).
\]
We will estimate from below the norm of $f_n$ in $\ell_{A}^{1}$. 

\subsection*{Step 1. {\it Application of Theorem \ref{thm:3} and summation
over a suitable range of $k$. }}
\medskip  
\phantom{A}
\medskip  

A direct computation shows that
\[
\int_{\alpha}^{\alpha^{-1}}\frac{{\rm d}t}{[(\alpha^{-1}-t)(t-\alpha)]^{1/4}}=(2\Gamma(3/4))^{2}\sqrt{\frac{\lambda}{\pi(1-\lambda^{2})}},
\]
and we can choose $\beta\in(\alpha,1)$ such that
$$
\int_{\beta}^{\beta^{-1}}\frac{{\rm d}t}{[(\alpha^{-1}-t)(t-\alpha)]^{1/4}}=\frac{L}{\sqrt{1-\lambda}},
$$
where $L=(\Gamma(3/4))^2$.

We define $I_{n}$ to be the following set of integers 
\[
I_{n}=\mathbb{Z}\cap[\beta n,\beta^{-1}n],
\]
and observe that, obviously, 
\[
\|f\|_{\ell_{A}^{1}}\geq\sum_{k\in I_{n}}\abs{\widehat{f}(k)}.
\]

A direct application of Theorem \ref{thm:3} gives
\begin{equation}
\sum_{k\in I_{n}}\abs{\widehat{f}(k)}\geq\sqrt{\frac{2}{n(1-\lambda^{2})\pi}}\sum_{k\in I_{n}}\frac{\abs{\cos(nF(k/n)-\psi(k/n)-\pi/4)}+o(1)}
{[(\alpha^{-1}-k/n)(k/n-\alpha)]^{1/4}}.
\label{h51}
\end{equation}
We set 
\begin{equation}
M=\sum_{k\in I_{n}}\frac{1}{\left[(\alpha^{-1}-k/n)(k/n-\alpha)\right]^{1/4}}. 
\label{h20}
\end{equation}
Comparing with the corresponding integral we obtain 
$$
M= n\int_{\beta}^{\beta^{-1}}\frac{{\rm d}t}{\left[(\alpha^{-1}-t)(t-\alpha)\right]^{1/4}}+\cO(1),\qquad n\to\infty.
$$
Therefore, 
\begin{equation}
M= \Bigl(\frac{L}{\sqrt{1-\lambda}}+o(1)\Bigr)n,\qquad n\to\infty. 
\label{h28}
\end{equation}

From now on, our aim is to estimate asymptotically 
\[
X=\sum_{k\in I_{n}}\frac{\abs{\cos\left(nF(k/n)-\psi(k/n)-\pi/4\right)}}{\left[(\alpha^{-1}-k/n)(k/n-\alpha)\right]^{1/4}}
\]
as $n$ grows large. We set
$$
s_{n,k}=\Bigl\{\frac{1}{2\pi}\bigl(nF(k/n)-\psi(k/n)-\pi/4\bigr)\Bigr\}\in[0,1),
$$
where $\bigl\{a\bigr\} $ denotes the fractional part of $a$, so that
\begin{equation}
X=\sum_{k\in I_{n}}\frac{|\cos (2\pi s_{n,k})|}{\left[(\alpha^{-1}-k/n)(k/n-\alpha)\right]^{1/4}}.
\label{h21}
\end{equation}

\subsection*{Step 2. {\it Equidistribution of }$s_{n,k}$. }
\medskip  
\phantom{A}
\medskip   

Given $j\in\mathbb Z\setminus\{0\}$ we set
\[
A_{n,j,m}=\sum_{l\in I_n,\,l\le m}\exp(2\pi ijs_{n,l})=\sum_{l\in I_n,\,l\le m}\exp(2\pi ij\phi(l/n)),
\]
where
\[
\phi=nF-\psi-\pi/4.
\]
We will show that, for fixed $j$ and uniformly in $m\in I_{n}$,  
\[
A_{n,j,m}=\cO(n^{1/2}),\qquad n\to\infty.
\]
To estimate $A_{n,j,m}$ we use a van der Corput lemma \cite[Ch. 5, Lemma 4.6]{ZA}:
if $\theta''(x)\geq\mu>0$ or $\theta''(x)\leq-\mu<0$ on $[a,b]$, then 
\begin{equation}
\Bigl|\sum_{a<k\le b}\exp(2\pi i\theta(k))\Bigr|\le \bigl(\abs{\theta'(b)-\theta'(a)}+2\bigr)\Bigl(\frac{4}{\sqrt{\mu}}+C\Bigr),\label{eq:VDCP}
\end{equation}
where $C$ is an absolute constant. We apply this lemma to the function
$\theta(x)=j\phi(x/n)$, and the numbers $a=\beta n$, $b=m\le \beta^{-1} n$. 

We have 
$$
\theta'(x)=\frac{j}{n}\phi'(\frac{x}{n})=jF'(x/n)-\frac{j}{n}\psi'(x/n)
$$ 
and 
$$
\theta''(x)=\frac{j}{n^2}\phi''(x/n)=\frac{j}{n}F''(x/n)-\frac{j}{n^2}\psi''(x/n).
$$

We set $t=x/n$. When $a\le x\le b$ we have $t\in(\beta,\beta^{-1})$.
Let us recall that 
\begin{gather}
q(t)=\frac{1+\lambda^2}{2\lambda}-\frac{1-\lambda^2}{2\lambda t},\notag\\
\varphi(t)=\arccos q(t),\notag\\
1-\lambda e^{i\varphi(t)}=|1-\lambda e^{i\varphi(t)}|e^{i\psi(t)},\qquad -\pi/2<\psi(t)<0,\label{seven7}\\
F(t)=\varphi(t)-t\varphi(t)-2\psi(t).\notag
\end{gather}

We define 
$$
r(t)=\sqrt{(t-\alpha)(\alpha^{-1}-t)}.
$$
Then
$$
r(t)=\frac{2\lambda t}{1-\lambda^2}\sqrt{1-q^2(t)},
$$
and
\begin{align}
q'(t)&=\frac{1-\lambda^2}{2\lambda t^2},\notag\\
r'(t)&=\frac1{r(t)}\Bigl(\frac{1+\lambda^2}{1-\lambda^2}-t\Bigr).\label{h8}
\end{align}
By the definition of $\varphi$, we have
$$
\varphi'(t)=-\frac{q'(t)}{\sqrt{1-q^2(t)}}=-\frac{1-\lambda^2}{2\lambda t^2\sqrt{1-q^2(t)}}=-\frac1{tr(t)}.
$$

Furthermore,
\begin{align*}
|1-\lambda e^{i\varphi(t)}
|&=\sqrt\frac{1-\lambda^2}{t},\\
1-\lambda e^{i\varphi(t)}
&=(1-\lambda^2)\frac{t+1}{2t}-i\lambda \sqrt{1-q^2(t)}.
\end{align*}

By the definition of $\psi$, we have 
$$
\tan \psi(t)=-\frac{2t\lambda\sqrt{1-q^2(t)}}{(t+1)(1-\lambda^2)}=-\frac{r(t)}{t+1}.
$$
Therefore, 
\begin{multline*}
\psi'(t)=\frac{1}{1+r^2(t)/(t+1)^2}\cdot\frac{r(t)-r'(t)(t+1)}{(t+1)^2}=\frac{r(t)-r'(t)(t+1)}{r^2(t)+(t+1)^2}\\
=\frac{r^2(t)-(t+1)(\frac{1+\lambda^2}{1-\lambda^2}-t)}{r(t)(r^2(t)+(t+1)^2)}=\frac{-1-t^2+2\frac{1+\lambda^2}{1-\lambda^2}t-t\frac{1+\lambda^2}{1-\lambda^2}-\frac{1+\lambda^2}{1-\lambda^2}+t^2+t}{r(t)(2\frac{1+\lambda^2}{1-\lambda^2}t+2t)}\\
=\frac{\frac{2}{1-\lambda^2}t-\frac{2}{1-\lambda^2}}{4r(t)t/(1-\lambda^2)}=\frac{t-1}{2r(t)t},
\end{multline*}
and 
$$
\psi''(t)=
\frac{r(t)+tr'(t)-t^2r'(t)}{2t^2r^2(t)}.
$$
 
Next, we have
$$
F'(t)=\varphi'(t)-\varphi(t)-t\varphi'(t)-2\psi'(t)
=-\varphi(t),
$$
and
$$
F''(t)=\frac{1}{tr(t)}.
$$

Since $t\in[\beta,\beta^{-1}]$, 
we have 
\begin{equation}
r(t)\asymp 1.
\label{h32}
\end{equation}
Here and below, since $\beta=\beta(\lambda)$ depends on $\lambda$, the implicit constants in \eqref{h32} (and in subsequent $\lesssim$/$\cO(\cdot)$ bounds that use it) may depend on $\lambda$, but are uniform in $t\in[\beta,\beta^{-1}]$ for fixed $\lambda$.

By \eqref{h8} we have
$$
|r'(t)|\lesssim \frac1{r(t)}.
$$
Therefore,
$$
\Bigl|\frac1j \theta'(x)\Bigr|=\Bigl|F'(t)-\frac{1}{n}\psi'(t)\Bigr|=\Bigl|\varphi(t)+\frac{t-1}{2ntr(t)}\Bigr|=\cO(1),
$$
and
$$
\frac{n}j \theta''(x)=F''(t)-\frac{1}{n}\psi''(t)=\frac{1}{tr(t)}\Bigl[1-\frac{r(t)+tr'(t)-t^2r'(t)}{2ntr(t)}\Bigr].
$$
Since
$$
\frac{|r(t)+tr'(t)-t^2r'(t)|}{2ntr(t)}=\cO_{\lambda}(n^{-1}),\qquad n\to\infty,
$$
for $n$ large enough we obtain that 
$$
\frac{n}j \theta''(x)\ge c(\lambda)>0, \qquad x\in[a,b],\,j\not=0.
$$
Now an application of van der Corput's inequality 
yields that, for fixed $\lambda$, $j\neq0$, and uniformly in $m\in I_{n}$, we have  
\begin{equation}
A_{n,j,m}=\cO(n^{1/2}),\qquad n\to\infty. 
\label{h30}
\end{equation}

\subsection*{Step 3. {\it Approximation by trigonometric polynomials
and Abel's transformation. }}
\medskip  
\phantom{A}
\medskip

Our argument here is analogous to that in the proof of \cite[Lemma 3]{GD}. Let 
$g(x)=|\cos(2\pi x)|$. We want to prove that 
\begin{equation}
\lim_{n\to\infty}\frac{X}{M} =\int_{0}^{1}g(x)\,{\rm d}x=\frac2\pi,
\label{h25}
\end{equation}
where $M$ is defined by \eqref{h20} and $X$ is given by \eqref{h21}. We observe that $g$
is continuous and 1-periodic on the real line. Fix $\varepsilon>0$.
By the Weierstrass approximation theorem, 
there exists a trigonometric
polynomial $p$,
\[
p(x)=\sum_{\abs{j}\leq N}c_{j}\exp(2\pi ijx)
\]
such that $\|g-p\|_{L^\infty[0,1]}\le\varepsilon$.
Therefore,
$$
\Bigl|X-\sum_{k\in I_{n}}\frac{p (s_{n,k})}{\left[(\alpha^{-1}-k/n)(k/n-\alpha)\right]^{1/4}}\Bigr|\le 
\varepsilon\sum_{k\in I_{n}}\frac{1}{\left[(\alpha^{-1}-k/n)(k/n-\alpha)\right]^{1/4}}=\varepsilon M,
$$
and
$$
\Bigl|\int_0^1g(x)\,{\rm d}x-c_0\Bigr|\le \varepsilon.
$$

Next,
$$
\sum_{k\in I_{n}}\frac{p (s_{n,k})}{\left[(\alpha^{-1}-k/n)(k/n-\alpha)\right]^{1/4}}=\sum_{k\in I_{n}}\frac{p (s_{n,k})}{\sqrt{r(k/n)}}=
\sum_{|j|\le N}c_{j}\sum_{k\in I_{n}}\frac{\exp(2\pi ij s_{n,k})}{\sqrt{r(k/n)}},
$$
and
$$
c_0\sum_{k\in I_{n}}\frac{1}{\sqrt{r(k/n)}}=c_0M.
$$

Therefore, to get \eqref{h25}, it remains to verify that for every fixed $j\ne 0$, we have 
$$
\sum_{k\in I_{n}}\frac{\exp(2\pi ij s_{n,k})}{\sqrt{r(k/n)}}=o(M),\qquad n\to\infty,
$$
or, equivalently, 
\begin{equation}
Y=\sum_{k\in I_{n}}\frac{A_{n,j,k}-A_{n,j,k-1}}{\sqrt{r(k/n)}}=o(M),\qquad n\to\infty.
\label{h40}
\end{equation}

By Abel's summation formula we have
$$
Y=
\frac{A_{n,j,\beta^{-1}n}}{\sqrt{r([\beta^{-1}n]/n)}}-
\sum_{k\in I_n}A_{n,j,k-1}\Bigl(\frac1{\sqrt{r(k/n)}}-\frac1{\sqrt{r((k-1)/n)}}\Bigr).
$$

By \eqref{h30}, we obtain that 
$$
|Y|\lesssim n^{1/2}
\Bigl(\frac{1}{\sqrt{r(\beta^{-1})}}+
\sum_{k\in I_n} \Bigl|\frac1{\sqrt{r(k/n)}}-\frac1{\sqrt{r((k-1)/n)}}\Bigr|\Bigr).
$$
Now we use \eqref{h32} and the fact that the function $\sqrt{r}$ is monotonic on two intervals whose union is $(\alpha n,\alpha^{-1}n)$. 
Therefore, 
$$
\frac{1}{\sqrt{r(\beta^{-1})}}+
\sum_{k\in I_n} \Bigl|\frac1{\sqrt{r(k/n)}}-\frac1{\sqrt{r((k-1)/n)}}\Bigr|=\cO(1),\qquad n\to\infty.
$$ 
Together with \eqref{h28} this gives \eqref{h40}, and, hence, \eqref{h25}.

\subsection*{Step 4. {\it Conclusion.}}
\medskip  
\phantom{A}
\medskip   

Relations \eqref{h51}, \eqref{h28}, and \eqref{h25} give us that
$$
\sum_{k\in I_n}|\widehat{f}(k)| \ge \frac{(L+o(1))\sqrt{n}}{\sqrt{1-\lambda}},\qquad n\to\infty.
$$
As a consequence,
\[
\|f\|_{\ell^1_A}\ge\kappa\frac{\sqrt{  n}}{1-\lambda},\qquad n\ge n(\lambda),
\]
for an absolute constant $\kappa>0$,
which completes the proof since $\|f\|_{\ell^2_A}\asymp\frac{1}{\sqrt{1-\lambda}}$. 
\end{proof}
\section{\label{sec:Proof-of-Theorem3}Proof of Theorem 3}

We first recall the definition of the Fourier/Taylor coefficients
of $f_{n}$: 
$$
\widehat{f_{n}}(k) 
=\frac{1}{2\pi i}\oint_{\mathbb{T}}\frac{b_{\lambda}^{n}(z)}{1-\lambda z}\,z^{-k}\,\frac{{\rm d}z}{z}
 =\frac{1}{2i\pi}\int_{\mathbb{T}}\frac{1}{1-\lambda z}\,e^{n\Phi(z)}\,\frac{{\rm d}z}{z}.
$$
To analyze their asymptotic behavior as $n\to+\infty$, we shall express this contour integral 
as the real part of an oscillatory integral
over $(0,\pi)$, after a passage to the conjugate form. In this regime
we apply a \textit{stationary phase theorem}, namely Fedoryuk\textquoteright s
result \cite[Theorem 1.6 p.107]{F1}, stated below in our notation.
%
%
%
Before proceeding to the detailed analysis, 
it is convenient
to recall the structure of the stationary points of the phase function
\[
\Phi(z)=\Phi_{t}(z)=\log\left(z^{-t}b_{\lambda}(z)\right),\qquad t=k/n,
\]
which governs the oscillatory behavior of the integral above. 
The following description of the stationary points was first established in~\cite[Lemma~11]{SzZa4} and later refined in~\cite[Lemma~5]{BFZ}. 
For completeness we restate it here.
\begin{lem} \textup{\cite[Lemma 11]{SzZa4} } \textup{\cite[Lemma 5]{BFZ} }
\label{lem:critical_pts_f} Let $t=k/n\in(\alpha,\,\alpha^{-1})$, let $\Phi$
be defined as above, and let $\varphi$ be defined by \eqref{h1}. 
Then $\Phi$ has two
distinct stationary points $z_{\pm}=e^{\pm i\varphi(t)}$ of order
one
and 
$$
\Phi''(z_{\pm})=
\frac{(1-\lambda^{2})(z_{\pm}-z_{\mp})\lambda}{(z_{\pm}-\lambda)^{2}(1-\lambda z_{\pm})^{2}}.
$$
\end{lem}

We now turn to the analysis of the Fourier/Taylor coefficients. 
We recall the classical 
stationary phase theorem due
to M.~V.~Fedoryuk \cite[Theorem 1.6 p.107]{F1}. 
For the reader's convenience we restate it
here in our present notation. 

\begin{thm}
\textup{\label{thm:Fedoryuk} \cite[Theorem 1.6 p.107]{F1} }Let $I=[a,b]\subset\mathbb R$, 
and let $t$ denote an auxiliary real parameter belonging to a domain
$\Omega\subset\mathbb{R}^{k}$. Assume that the functions 
\[
G:\;I\times\Omega\to\mathbb C,\qquad\tilde{h}:\;I\times\Omega\to\mathbb R,
\]
satisfy the following hypotheses: 

{\rm(i)} $G$ is $C^{\infty}$ smooth on $I\times\Omega$ and, together
with all its derivatives in the first variable, vanishes at the endpoints
$a,b$.

{\rm(ii)} For every fixed $t\in\Omega$, the phase function $\tilde{h}_{t}(s)=\tilde{h}(s,t)$
is real-valued and possesses a unique non-degenerate stationary point
$s_0=s_0(t)\in(a,b)$, satisfying 
\[
\tilde{h}_{t}'(s_0(t))=0,\qquad|\tilde{h}_{t}''(s_0(t))|\ge\delta_{0}>0,\qquad t\in\Omega.
\]

Then, the oscillatory integral 
\[
J(n,t)=\int_a^b G(s,t)\,e^{\,in\tilde{h}_{t}(s)}\,ds,
\]
possesses the asymptotic expansion as $n\to+\infty$, 
\[
J(n,t)=\sqrt{\frac{2\pi}{n\,|\tilde{h}_{t}''(s_0(t))|}}\,\Bigl(G(s_0(t),t)+O(n^{-1})\Bigr)\exp\Bigl\{\,in\tilde{h}_{t}(s_0(t))+\frac{i\pi}{4}\,{\rm sgn}\bigl(\tilde{h}_{t}''(s_0(t))\bigr)\Bigr\},
\]
where the remainder term $O(n^{-1})$ is uniform over $t$ in every compact
subset of $\Omega$. 
\end{thm}

Going back to the Fourier/Taylor coefficients of $f_{n}$ we observe
that the latter are real (because $\lambda\in(0,1)$). Thus, we have
$$
\widehat{f_{n}}(k)  =\frac{1}{\pi}\Re\Bigl\{\int_{0}^{\pi}\frac{1}{1-\lambda e^{is}}e^{inh_{t}(s)}\,{\rm d}s\Bigr\}
 =\frac{1}{\pi}\Re\Bigl\{\int_{0}^{\pi}\frac{1}{1-\lambda e^{-is}}e^{in\tilde{h_{t}}(s)}\,{\rm d}s\Bigr\},
$$
where $t=k/n$, 
\[
h(s)=h_{t}(s)=-i\Phi_{t}(e^{is}),\qquad 
\tilde{h}(s)=\tilde{h}_t(s):=-h_{t}(s).
\]


By \eqref{h1xx} and \eqref{h1} we have 
\[
\Abs{e^{i\varphi(t)}-1}\ge(1-\lambda)\sqrt{\frac{\beta-\alpha}{\alpha\lambda}},\qquad\Abs{e^{i\varphi(t)}+1}\geq(1+\lambda)\sqrt{\frac{\beta-\alpha}{\beta\lambda}}.
\]
Therefore, given $\beta\in(\alpha,1)$, we can fix $\gamma=\gamma(\alpha,\beta)>0$ such that 
if $t\in[\beta,\beta^{-1}]$, then the only critical point $\varphi(t)$ of $\tilde{h}_t$ on $(0,\pi)$ satisfies the inequality $\gamma\le \varphi(t)\le\pi-\gamma$. 
In view of applying Fedoryuk's result we introduce a smooth function $\nu:\left[0,\pi\right]\to\mathbb R$
such that $\nu=1$ on $[\gamma/2,\pi-\gamma/2],$ $\nu=0$ on
$[0,\gamma/4]\cup[\pi-\gamma/4,\pi]$ and $0\le\nu\le 1$.
Setting 
\[
g(s)=\frac{1}{1-\lambda e^{-is}},
\]
we have 
\begin{gather*}
\int_0^\pi g(s)e^{in\tilde{h}(s)}{\rm d}s  =\int_0^\pi(1-\nu(s))g(s)e^{in\tilde{h}(s)}{\rm d}s+\int_0^\pi\nu(s)g(s)e^{in\tilde{h}(s)}{\rm d}s\\
  =\int_0^{\gamma/2}\!\!\!(1-\nu(s))g(s)e^{in\tilde{h}(s)}{\rm d}s
  +\int_{\pi-\gamma/2}^\pi\!\!\!\!(1-\nu(s))g(s)e^{in\tilde{h}(s)}{\rm d}s
  +\int_{\gamma/4}^{\pi-\gamma/4}\!\!\!\!\nu(s)g(s)e^{in\tilde{h}(s)}{\rm d}s.
\end{gather*}
Integrating by parts and taking into account that $\tilde{h}'$ does not
vanish on $[0,\gamma/2]$, we obtain that  
\begin{gather*}
\int_0^{\gamma/2}(1-\nu(s))g(s)e^{in\tilde{h}(s)} {\rm d}s =\frac{1}{in}\int_{0}^{\gamma/2}\frac{(1-\nu(s))g(s)}{\tilde{h}'(s)}
(e^{in\tilde{h}(s)})'{\rm d}s\\
  =\frac{(1-\nu(\gamma/2))g(\gamma/2)e^{in\tilde{h}(\gamma/2)}}{in\tilde{h}'(\gamma/2)}-\frac{(1-\nu(0))g(0)e^{in\tilde{h}(0)}}{in\tilde{h}'(0)}
  \\-\frac{1}{in}\int_{0}^{\gamma/2}
  \Bigl(\frac{(1-\nu(s))g(s)}{\tilde{h}'(s)}\Bigr)'e^{in\tilde{h}(s)}{\rm d}s\\
  =-\frac{1}{in}\int_{0}^{\gamma/2}
  \Bigl(\frac{(1-\nu(s))g(s)}{\tilde{h}'(s)}\Bigr)'e^{in\tilde{h}(s)}{\rm d}s
  -\frac{e^{in\tilde{h}(0)}}{in(1-\lambda)\tilde{h}'(0)},
\end{gather*}
because $\nu(\gamma/2)=1$ and $g(0)=\frac{1}{1-\lambda}$. By \eqref{four4}, 
$$
\min_{\varphi\in[0,\gamma/2]}|\tilde{h}'|=|\tilde{h}'(\gamma/2)|=-\tilde{h}'(\gamma/2), 
$$
and by the mean value theorem there exists $\theta\in(\gamma/2,\varphi(t))$
such that 
\begin{align*}
-\tilde{h}'(\gamma/2) & =\tilde{h}'(\varphi(t))-\tilde{h}'(\gamma/2)\\
 & =(\varphi(t)-\gamma/2)\tilde{h}''(\theta)\ge \frac{\gamma}{2}\frac{2\lambda(1-\lambda^{2})\sin\theta}{(1+\lambda^{2}-2\lambda\cos\theta)^{2}}\\
 & \ge\frac{\gamma\lambda(1+\lambda)}{1-\lambda}\min_{t\in[\gamma/2,\pi-\gamma]}\sin\theta>0.
\end{align*}
Therefore, 
\[
\int_{0}^{\gamma/2}(1-\nu(s))g(s)e^{in\tilde{h}(s)}{\rm d}s=\cO(n^{-1})
\]
uniformly for $t=\frac{k}{n}\in[\beta,\beta^{-1}]$. An analogous argument gives 
\[
\int_{\pi-\gamma/2}^{\pi}(1-\nu(s))g(s)e^{in\tilde{h}(s)}{\rm d}s=\cO(n^{-1})
\]
uniformly for $t\in[\beta,\beta^{-1}]$. Thus, 
\begin{equation}
\widehat{f_{n}}(k)=\frac{1}{\pi}\Re\int_{\gamma/4}^{\pi-\gamma/4}\nu(s)g(s)e^{in\tilde{h}(s)}{\rm d}s+\cO(n^{-1}).\label{eq:approx_f_nhat}
\end{equation}
It remains to apply Fedoryuk's result to 
\[
J(n,t)=\int_{\gamma/4}^{\pi-\gamma/4}\nu(s)g(s)e^{in\tilde{h}(s)}{\rm d}s.
\]
We set
\[
I=[a,b]=[\gamma/4,\pi-\gamma/4],\quad\Omega=[\beta,\beta^{-1}],
\]
and 
\[
G(s,t)=\frac{\nu(s)}{1-\lambda e^{-is}}.
\]
As discussed above, for $t\in[\beta,\beta^{-1}]$
the unique critical point $\varphi(t)$ of the function $\tilde{h}_t$ on $(0,\pi)$ 
satisfies the inequality 
\[
\tilde{h}_t''(\varphi(t))\ge C(\beta,\lambda)>0,
\]
which allows us to apply Fedoryuk's asymptotic formula. 
We have $\nu(\varphi(t))=1$ because $\varphi(t)\in[\gamma,\pi-\gamma]$. 
By \eqref{four4}, 
\[
\tilde{h}_t''(\varphi(t))=tr(t),
\]
and we obtain 
\begin{align*}
J(n,t) & =\sqrt{2\pi}\abs{\tilde{h}_t''(\varphi(t))}^{-1/2}e^{i\frac{\pi}{4}}e^{in\tilde{h}_t(\varphi(t))}{\nu(\varphi(t))}g(\varphi(t))n^{-1/2}+\cO\left(n^{-3/2}\right)\\
 & =e^{in\tilde{h}_t(\varphi(t))+i\frac{\pi}{4}}\frac{1}{1-\lambda e^{-i\varphi(t)}}\frac{1}{t^{1/2}(t-\alpha)^{1/4}(\alpha^{-1}-t)^{1/4}}\Bigl(\frac{2\pi}{n}\Bigr)^{1/2}+\cO(n^{-3/2}),
\end{align*}
where the $\cO$-term is uniform over $t\in[\beta,\beta^{-1}]$. 

By \eqref{seven7}, we have 
\begin{multline*}
1-\lambda e^{-i\varphi(t)}=\bigl((1-\lambda\cos (\varphi(t)))^2+\lambda^2(1-\cos(\varphi(t)))^2\bigr)^{1/2}e^{-i\psi(t)}\\
=\bigl(1+\lambda^2-2\lambda q(t)\bigr)^{1/2}e^{-i\psi(t)}=\sqrt{\frac{1-\lambda^{2}}{t}}e^{-i\psi(t)}.
\end{multline*}

Therefore,
$$
J(n,t) =e^{in\tilde{h}_t(\varphi(t))+i\frac{\pi}{4}+i\psi(t)}\frac{1}{(t-\alpha)^{1/4}(\alpha^{-1}-t)^{1/4}}\Bigl(\frac{2\pi}{n(1-\lambda^{2})}\Bigr)^{1/2} +\cO(n^{-3/2}),\,\, n\to\infty,
$$
with the $\cO$-term uniform over $t\in[\beta,\beta^{-1}]$. 
Going
back to \eqref{eq:approx_f_nhat} 
we conclude 
\begin{align*}
\widehat{f_{n}}(k) & =\sqrt{\frac{2}{n(1-\lambda^{2})\pi}}\frac{\cos\bigl(n\tilde{h}_t(\varphi(t))+\frac{\pi}{4}+\psi(t)\bigr)}{[(\alpha^{-1}-t)(t-\alpha)]^{1/4}}+\cO\left(n^{-1}\right)\\
 & =\sqrt{\frac{2}{n(1-\lambda^{2})\pi}}\frac{\cos(nF(t)-\frac{\pi}{4}-\psi(t))}{[(\alpha^{-1}-t)(t-\alpha)]^{1/4}}+\cO\left(n^{-1}\right),
\end{align*}
where the $\cO$-term is uniform over $t\in[\beta,\beta^{-1}]$.
%
%
%
%
%

%
%
%
%
%

\end{document}